\documentclass[a4paper,10pt]{article}
\usepackage{amsthm,amssymb,amsmath,enumerate,graphicx,epsf,bbold}

\newcommand{\COLORON}{1}
\newcommand{\NOTESON}{0}
\newcommand{\Debug}{0}
\usepackage[usenames,dvips]{color} 
\usepackage{amsthm,amssymb,amsmath,enumerate,graphicx,epsf,mathbbol,stmaryrd}
\usepackage[dvips, bookmarks, colorlinks=false, breaklinks=true]{hyperref}

\newcommand{\iicon}{2-connected}
\newcommand{\tcon}{3-connected}

\newcommand{\pr}{consistent}
\newcommand{\prem}{\pr\ embedding}
\newcommand{\vapf}{VAP-free}
\newcommand{\ccfg}{\ensuremath{\cc_f(\G)}}

\newcommand{\comment}[1]{}
\newcommand{\COMMENT}[1]{}

\definecolor{darkgray}{rgb}{0.3,0.3,0.3}
\newcommand{\defi}[1]{{\color{darkgray}\emph{#1}}}

\newcommand{\acknowledgements}{\section*{Acknowledgements}}

\comment{
	\begin{lemma}\label{}	
\end{lemma}
\begin{proof}

\end{proof}

\begin{theorem}\label{}
\end{theorem} 
\begin{proof} 	

\end{proof}

}

\newtheorem{proposition}{Proposition}[section]
\newtheorem{definition}[proposition]{Definition}
\newtheorem{theorem}[proposition]{Theorem}
\newtheorem{corollary}[proposition]{Corollary}
\newtheorem{lemma}[proposition]{Lemma}

\newtheorem*{noTheorem}{Theorem}
\newtheorem{examp}[proposition]{Example}

\newcommand{\FIG}{0}

\ifnum \NOTESON = 1 \newcommand{\note}[1]{  

	{\color{blue} \hspace*{-60pt} NOTE: \color{Turquoise}{\small  \tt \begin{minipage}[c]{1.1\textwidth}  #1 \end{minipage} \ignorespacesafterend }} 
	
	}
\else \newcommand{\note}[1]{} \fi

\newcommand{\afsubm}[1]{ \ifnum \Debug = 1 {\mymargin{#1}}
\fi} 

\ifnum \Debug = 1 
\else  \fi

\ifnum \FIG = 1 \newcommand{\fig}[1]{Figure ``{#1}''}
\else \newcommand{\fig}[1]{Figure~\ref{#1}} \fi

\ifnum \Debug = 1 \usepackage[notref,notcite]{showkeys}
\fi

\ifnum \COLORON = 0 \renewcommand{\color}[1]{}
\fi

\newcommand{\showFig}[2]{
   \begin{figure}[htbp]
   \centering
   \noindent
   \epsfbox{#1.eps}
   \caption{\small #2}
   \label{#1}
   \end{figure}
}

\newcommand{\R}{\ensuremath{\mathbb R}}
\newcommand{\Z}{\ensuremath{\mathbb Z}}

\newcommand{\ca}{\ensuremath{\mathcal A}}
\newcommand{\cb}{\ensuremath{\mathcal B}}
\newcommand{\cc}{\ensuremath{\mathcal C}}

\newcommand{\cf}{\ensuremath{\mathcal F}}

\newcommand{\cs}{\ensuremath{\mathcal S}}

\newcommand{\oo}{\ensuremath{\omega}}

\newcommand{\Gam}{\ensuremath{\Gamma}}

\newcommand{\sig}{\ensuremath{\sigma}}

\newcommand{\sm}{\backslash}

\newcommand{\pth}[2]{\ensuremath{#1}\text{--}\ensuremath{#2}~path}

\newcommand{\pths}[2]{\ensuremath{#1}\text{--}\ensuremath{#2}~paths}

\newcommand{\g}{\ensuremath{G\ }}
\newcommand{\G}{\ensuremath{G}}

\newcommand{\Lr}[1]{Lemma~\ref{#1}}
\newcommand{\Tr}[1]{Theorem~\ref{#1}}
\newcommand{\Sr}[1]{Section~\ref{#1}}

\newcommand{\Cr}[1]{Corollary~\ref{#1}}

\newcommand{\lf}{locally finite}

\newcommand{\Cg}{Cayley graph}

\renewcommand{\iff}{if and only if}
\newcommand{\fe}{for every}
\newcommand{\Fe}{For every}

\newcommand{\st}{such that}

\newcommand{\ti}{there is}
\newcommand{\ta}{there are}

\newcommand{\obda}{without loss of generality}

\newcommand{\wrt}{with respect to}

\newcommand{\ises}{is easy to see}

\newcommand{\labtequ}[2]{ \begin{equation} \label{#1} 	\begin{minipage}[c]{0.9\textwidth}  #2 \end{minipage} \ignorespacesafterend \end{equation} }

\newcommand{\mymargin}[1]{
  \marginpar{%
    \begin{minipage}{\marginparwidth}\small%
      \begin{flushleft}%
        {\color{blue}#1}%
      \end{flushleft}%
   \end{minipage}%
  }%
}%

\newcommand{\mySection}[2]{}

\newcommand{\sfed}{simplified}
\newcommand{\fapr}{facial presentation}
\newcommand{\ccf}{\ensuremath{\cc_f}} 
\newcommand{\ccfh}{\ensuremath{H_1}} 
\renewcommand{\vapf}{accumulation-free}
\newcommand{\Vapf}{Accumulation-free}
\title{Characterising planar \Cg s and Cayley complexes in terms of group presentations}
\author{Agelos Georgakopoulos\thanks{Partly supported by FWF grant P-19115-N18.} \medskip 
\\
  {Mathematics Institute}\\
 {University of Warwick}\\
  {CV4 7AL, UK}\\}

\begin{document}
\maketitle

\newtheorem*{nocorollary}{Corollary}

\begin{abstract}
We prove that a \Cg\ can be embedded in the euclidean plane without accumulation points of vertices \iff\ it is the 1-skeleton of a Cayley complex that can be  embedded in the plane after removing redundant simplices. We also give a characterisation of these \Cg s in term of group presentations, and deduce that they can be effectively enumerated.
\end{abstract}

\section{Introduction}
The study of groups that have \Cg s embeddable in the euclidean plane $\R^2$, called \defi{planar groups}, has a tradition starting in 1896 with Maschke's characterization of the finite ones. Among the infinite planar groups, those that admit a \defi{flat} Cayley complex, defined below, have received a lot of attention. They are important in complex analysis as they include the discontinuous groups of motions of the euclidean and hyperbolic plane. Moreover, they are closely related to surface groups \cite[Section 4.10]{ZVC}. These groups are now well understood due to the work of Macbeath \cite{macCla}, Wilkie \cite{wilNon}, and others; see \cite{ZVC} for a survey\footnote{In \cite{ZVC} the term  {\it Cayley complex} is not used but it is implicit in Theorems 4.5.6 and 6.4.7 that a group admits a flat Cayley complex if and only if it is a \defi{planar discontinuous group}.}. Planar groups that have no flat Cayley complex are harder to analyse, and they are the subject of on-going research  \cite{droInf,DrSeSeCon,dunPla,cayley3,cay2con}. 

All groups, \Cg s and Cayley complexes in this paper are finitely generated. Our first result is

\begin{theorem} \label{thmvapf}
A planar \Cg\ of a group \Gam\ is \vapf\ \iff\ it is the 1-skeleton of a flat Cayley complex of \Gam.
\end{theorem}

Here, a Cayley complex is \defi{flat} if it can be embedded in $\R^2$ after removing redundant 2-simplices; see \Sr{defflat} for the precise definition.
A planar graph is said to be \defi{\vapf}, if it admits an embedding  in $\R^2$ \st\ the images of its vertices have no accumulation point. 
The study of a planar graph is often simplified if one knows that the graph is \vapf; examples range from structural graph-theory \cite{CWY} to percolation theory  \cite{kozPPP} and the study of spectral properties \cite{keller}.  A further example is Thomassen's \Tr{thom} below, which becomes false in the non-\vapf\ case. \vapf\ graphs can be characterized by a condition similar to that of Kuratowski's; see \cite{halVAP}. {\em \Vapf} embeddings also appear with other names in the literature, most notably ``{\em \lf}''.

\Tr{thmvapf} implies that a group has a flat Cayley complex \iff\ it has an \vapf\ \Cg, a fact that might be known to experts, and it should not be too hard to derive it from the results of \cite{ZVC}. 
\Tr{thmvapf} however strengthens this assertion into a theorem about all planar \Cg s, not just their groups. Since a single group can have a large variety of planar \Cg s (see \Sr{sinv} for some examples), it is in principle harder to prove results that hold for all planar \Cg s than proving the corresponding result for their groups. However, our proof is elementary and self-contained, avoiding the geometric machinery
of \cite{ZVC}.

We also prove that every \vapf\ \Cg\ admits an embedding the facial walks of which are preserved by the action of the group; see \Cr{corspin}.

Finally, we derive a further characterisation of the \vapf\ \Cg s, and so by \Tr{thmvapf} also of the groups that admit a flat Cayley complex, by means of group presentations. We introduce a special kind of presentation, called a \defi{\fapr}, which is motivated by geometric intuition and can be easily recognised by an algorithm, and use it to obtain a further characterisation of the class of \vapf\ \Cg s:

\begin{corollary}
A \Cg\ admits an \vapf\ embedding \iff\ it admits a \fapr.
\label{corfapr}
\end{corollary} 

This implies that the \vapf\ \Cg s can be effectively enumerated (\Cr{efen}).

We prove \Tr{thmvapf} in \Sr{smain}. In \Sr{sinv} we examine \vapf ness  as a group-theoretical invariant. Finally, in \Sr{sML} we introduce \fapr s and prove \Cr{corfapr}. 

\section{Preliminaries}

We will follow the terminology of \cite{diestelBook05noEE} for graph-theoretical terms and that of \cite{bogop,hatcher} for group-theoretical ones. 

Let us recall some standard definitions used in this paper. We say that a graph \G\ is \defi{$k$-connected} if $G - X$ is connected for every set $X\subseteq V$ with $|X | < k$. A \defi{component} of \g is a maximal connected subgraph of \G. 

A \defi{walk} in \g is an alternating sequence $v_0e_0v_1e_1 \ldots e_{k-1}v_k$ of vertices and edges in \G\ such that $e_i = \{v_i,v_{i+1}\}$ for all $i < k$. If $v_0 = v_k$, the walk is \defi{closed}. If the vertices in a walk are all distinct, it is called a \defi{path} (many authors use the word `path' to denote a walk in our sense).

A $1$-way infinite path is
called a \defi{ray}, a $2$-way infinite path is
a \defi{double ray}. 
Two rays contained in a graph $G$ are \defi{equivalent} if no finite set of edges
separates them. The corresponding equivalence
classes of rays are the \defi{ends} of $G$.

By an \defi{embedding} of a graph \g we mean a topological embedding of the corresponding 1-complex in the euclidean plane $\R^2$; in simpler words, an embedding is a drawing of the graph in the plane with no two edges crossing. A graph is \defi{planar} if it admits an embedding. A \defi{plane} graph is a (planar) graph endowed with a fixed embedding.

A \defi{face} of an embedding $\sig: G \to \R^2$ is a component of $\R^2 \sm \sig(G)$. The \defi{boundary} of a face $F$ is the set of vertices and edges of \g that are mapped by \sig\ to the closure of $F$. A path, or walk, in \g is called \defi{facial} \wrt\ \sig\ if it is contained in the boundary of some face of \sig.

One of our main tools will be the (finitary) \defi{cycle space} \ccfg\ of a graph $G=(V,E)$, which is defined as the vector space over $\Z_2$ (the field of two elements) consisting of those subsets of $E$ such that can be written as a sum (modulo 2) of a finite set of {circuits}, where a set of edges $D\subseteq E$ is called a \defi{circuit} if it is the edge set of a cycle of \G. 

The cycle space is closely related to the first (simplicial) homology group $\ccfh(G)$ 
\cite{hatcher}, and in fact the two objects coincide when the latter is defined over the field $\Z_2$. In this paper $\ccfh(G)$ will be defined over $\Z$ as usual, and so it should not be confused with \ccfg.

\subsection{(Flat) Cayley complexes} \label{defflat}

The \defi{Cayley complex} $X$ of a group presentation $P= (S,R)$ is  the universal cover of its \defi{presentation complex}, which is a 2-dimensional cell complex with a single vertex,  one loop at the vertex for each generator in $S$, and one 2-cell for each relation in $R$ bounded by the loops corresponding to the generators appearing in $R$, see \cite{hatcher}. The 1-skeleton of this Cayley complex is the Cayley graph corresponding to the group presentation $(S,R)$. From the Cayley complex we derive the  \defi{\sfed} Cayley complex of $P$ as follows. Firstly, \fe\ pair of parallel edges $e,e'$, resulting from an involution in $S$, we identify $e$ with $e'$, gluing them together according to a homeomorphism from $e$ to $e'$ that maps each endvertex to itself. We remove any 2-simplices of $X$ that were bounded by the circle $e \cup e'$; all other 2-simplices incident with $e$ or $e'$ are preserved. 
In the resulting 2-complex $X'$, we define two 
 2-simplices to be \defi{equivalent} if they have the same boundary.  Removing all but one of the elements of each equivalence class from $X'$ we obtain the  \defi{\sfed} Cayley complex of $P$. 
  
Equivalently, we can define the {\sfed\ Cayley complex} of  $P= (S,R)$ by building the corresponding \Cg, identifying each pair of  parallel edges into a single edge, and then for every cycle $C$ of this graph induced by a relator in $R$, introducing a 2-simplex having $C$ as its boundary.

\begin{definition} We say that a Cayley complex is \defi{flat}, if the corresponding \sfed\ Cayley complex is \defi{planar}, that is, the latter admits an embedding into  $\R^2$.
\end{definition} 

For example, the Cayley complex of the  group presentation $\left<a \mid a^n \right>$ has $n$ equivalent 2-simplices, while the corresponding \sfed\ Cayley complex has only one, and is {planar}. As a further example, consider the   presentation $\left<a,b \mid b^2, aba^{-1}b\right>$. The corresponding \sfed\ Cayley complex is a 2-way infinite ladder with each 4-gon bounding a 2-simplex (\fig{ladder}); note that this complex is planar, while the usual Cayley complex is not (to see this, note that for each $b$ edge, \ta\ two distinct 4-cycles in the \Cg\ induced by the relation $aba^{-1}b$ and one 2-cycle induced by $b^2$). These two examples show that the above simplifications of the Cayley complex are necessary to make \Tr{thmvapf} true.

\showFig{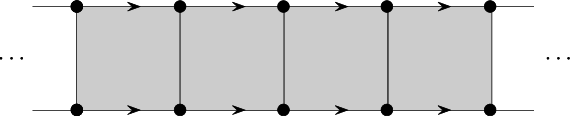}{The \sfed\ Cayley complex of the   presentation $\left<a,b \mid b^2, aba^{-1}b\right>$.}

\section{Proof of \Tr{thmvapf}} \label{smain}

In this section we will be assuming that our graphs have no parallel edges. In a Cayley graph this can be achieved by drawing, for every involution in the generating set, a single undirected edge rather than a pair of parallel edges with opposite directions. This convention affects neither planarity nor \vapf ness, and so our assumption comes without loss of generality for the proof of \Tr{thmvapf}.

Our first lemma, a well-known fact which is easy to prove, relates $\ccfh(G)$ to group presentations. We will say that a closed walk $W$ in \g is \defi{induced} by a relator $R$, if $W$ can be obtained by starting at some vertex $g$ and following the edges corresponding to the letters in $R$ in order; note that for a given $R$ \ta\ several walks in \g induced by $R$, one for each starting vertex $g\in V(G)$. Note that every closed walk in \g uniquely determines an element of $\ccfh(G)$, and we will, with a slight abuse, not make a distinction between the two.

\begin{lemma} \label{relcc}
Let $G= Cay(\Gam,S)$ be a \Cg\ of the group \Gam, and let $ \left< S \mid R\right> $ be a presentation of \Gam. Then the set of walks in \g induced by relators in $R$ generates $\ccfh(G)$. 

Conversely, if $R'$ is a set of relations of \Gam\  with letters in a generating set $S$ such that the set of closed walks of $Cay(\Gam,S)$ induced by $R'$ generates $\ccfh(G)$, then $ \left< S \mid R' \right> $ is a presentation of \Gam.
\end{lemma}

Combined with the next easy fact, this allows one to deduce group presentations from \vapf\ embeddings of a \Cg.

\begin{lemma} \label{facecc}
Let $G$ be an \vapf\ plane graph. Then the set of finite facial closed walks of \g generates $\ccfh(G)$. 
\end{lemma}
\begin{proof}
It suffices to show that every cycle $C$ of \g is a sum of  finite facial closed walks  when seen as an element of $\ccfh(G)$. This is indeed the case, for as \g is \vapf\ there must be a side $A$ of $C$ containing only finitely many vertices, and so $E(C)$ is the sum of the facial closed walks corresponding to faces lying in $A$.
\end{proof}

We will also use the following basic characterisation of \vapf\ graphs
\begin{lemma}[{\cite[Lemma~7.1]{thoPla}}]  \label{vapfeq}
A countable graph \g is \vapf\ \iff\ some planar embedding of \g has the property that no cycle has both infinitely many vertices in its interior and infinitely many vertices in its exterior. 
\end{lemma}

The following fact is probably known to experts in the study of infinite vertex transitive graphs. We include a proof sketch for the convenience of the non-expert. A \defi{double-ray} is a 2-way infinite path (with no repetition of vertices).
\begin{lemma} \label{sepfin}
Let \g be an infinite, connected, vertex transitive graph which is not a double-ray. Then for every pair of vertices $x,y$ of \G, no component of $G -\{x,y\}$ is finite.
\end{lemma}
\begin{proof}
To begin with, it is easy to prove that 
\labtequ{x}{\fe\ $x\in V(G)$, no component of $G -\{x\}$ is finite,}
by considering a minimal such component $C$ and mapping $x$ to some vertex of $C$. 

Suppose that some component $C$ of $G -\{x,y\}$ is finite, and choose $x,y$ so as to minimise $|V(C)|$. We claim that the graph $C$ has no cut-vertex. Indeed, if $z\in V(C)$ separates $C$, then $G -\{x,z\}$ contains a component properly contained in $C$, contradicting the minimality of the latter. Moreover, each of $x,y$ has at least two neighbours in $C$; for if $y$ has a single neighbour $y'$ in $C$, then we could have replaced $y$ by  $y'$ to obtain a separator $\{x,y'\}$ cutting off a smaller component, and if $y$ has no neighbour in $C$ then \eqref{x} is contradicted. These two observations, combined with Menger's theorem \cite[Theorem~3.3.1]{diestelBook05noEE}, imply that there are two independent \pths{x}{y}\ $P,Q$ through $C$. Moreover, (at least) one of $x,y$, say $x$, is contained in an infinite subgraph $X$ that does not meet $C \cup \{x,y\}$ except at $x$.


Let $z\in V(C)$, and consider an automorphism $g$ mapping $x$ to $z$. Then, \ti\ a vertex $w=gy$ \st\ $\{z,w\}$ separates \G. We consider three cases.

If $w$ lies in $C':= G \sm (C \cup \{x,y\})$, then each of $gP, gQ, gX$ meets both $C'$ and $C$. But this is impossible since $C$ is separated from $C'$ by $x,y$ and the only vertices meeting more than one of $gP, gQ, gX$ are $z$ and $w$, none of which equals $x$ or $y$.

If $w$ lies in $C$, then some component of $G -\{z,w\}$ is properly contained in $C$ contradicting its minimality.

Finally, if $w=y$, then as the component $gC$ of $G -\{z,w\}$ cannot be smaller than $C$, it must contain a vertex $x'$ in $G \sm (C \cup \{y\})$. Note that there is a \pth{x'}{x}\ $P$ in $G \sm (C \cup \{y\})$, because otherwise $x'$ lies in a component of $G-\{x,y\}$ sending no edges to $x$, contradicting \eqref{x}. Recall that the infinite subgraph $X$ mentioned above does not meet $C \cup \{x,y\}$ except at $x$. Thus the infinite subgraph $X \cup P$ also does not meet $C \cup \{x,y\}$ except at $x$. Since this subgraph meets $gC$ (at the vertex $x'$), it is contained in $gC$ which is separated from the rest of the graph by $\{z,y\}$. This shows  that $gC$ is infinite, contradicting the fact that it is a translate of the finite $C$. 

Thus, in all three cases we obtained a contradiction. This proves \Lr{sepfin}.
\end{proof}

We can now prove \Tr{thmvapf}, which we restate for the convenience of the reader

\begin{noTheorem}
A planar \Cg\ of a group \Gam\ is \vapf\ \iff\ it is the 1-skeleton of a flat Cayley complex of \Gam.
\end{noTheorem}
\begin{proof}
For the forward implication, let $G$ be a planar \Cg\ of the group \Gam, \wrt\ the generating set $\cs$, admitting an \vapf\ embedding \sig. By \Lr{faces} below, if $F$ is a finite face boundary in \sig, then every translate of $F$ is a face boundary. This means that if we let ${\mathcal R}'$ be the set of relations corresponding to the finite facial walks \wrt\ \sig\ incident with the group identity $e$, then every finite facial walk \wrt\ \sig\ is induced by some element of ${\mathcal R}'$, and conversely any cycle induced by some element of ${\mathcal R}'$ bounds a face in \sig. By Lemmas \ref{facecc} and \ref{relcc}, $\left< \cs \mid {\mathcal R'} \right>$ is a presentation of \Gam. The corresponding \sfed\ Cayley complex is planar and \vapf\ since we can embed its 1-skeleton \g\ by \sig\ and then every 2-simplex can be embedded into the face of \sig\ bounded by the corresponding cycle.

For the backward implication, 
let $X$ be a planar \sfed\ Cayley complex and let \g be its 1-skeleton. Let $B$ be the set of closed walks in \g  bounding a 2-simplex of $X$; in fact, each such closed walk is a cycle since $X$ is planar, and  it bounds a face of \G.
Note that $B$ generates  $\ccfh(G)$ by \Lr{relcc} and the definition of a Cayley complex. We will show that the condition in \Lr{vapfeq} is satisfied, i.e.\ no cycle of \g  has both infinitely many vertices in its interior and infinitely many vertices in its exterior. Indeed, every cycle $K$ can be written as a finite sum of elements of $B$ since the latter generates  $\ccfh(G)$. As each element of $B$ bounds a face of \G, it is not hard to see that this sum comprises the face boundaries in the interior of $K$. This implies that the interior of $K$ contains only finitely many vertices. Thus by \Lr{vapfeq}, \g is \vapf.

\end{proof}

A \defi{translate} of a subgraph $F$ of \g is the image of $F$ under an automorphism of \G.
In the following lemma our assumption that \g has no parallel edges becomes essential.
\begin{lemma} \label{faces}
Let \g be a vertex transitive graph with an \vapf\ embedding \sig. If $F$ is a finite face boundary in \sig\ then every translate  of $F$  is a face boundary in \sig.
\end{lemma}
\begin{proof}
Suppose to the contrary that some image $F'=gF$ of $F$ under an automorphism $g$ is not a face boundary. Then, as \sig\ is \vapf, one of the sides of $F'$ contains at least one finite \defi{bridge} $C$ of $F'$, where by a bridge  of $F'$ we mean either a finite component of $G - F'$ or an edge joining two vertices of $F'$. Let $N(C)$ be the set of vertices of $F'$ incident with $C$ (if $C$ is an edge, then $N(C)$ are its endvertices). Then $F' - N(C)$ consists of a set of disjoint paths, which we call the \defi{intervals}. Note that unless $C$ is an edge, we have $|N(C)|\geq 3$ for otherwise \Lr{sepfin} is contradicted as $N(C)$ separates $C$.

We claim that
\labtequ{interv}{no bridge of $F'$ is adjacent with more than one interval.}
Indeed, if such a bridge $C'$ existed, then, by a topological argument, it would be impossible to embed \g\ is such a way that both $C$ and $C'$ lie in the same side of $F'$ (\fig{fIJ}), but such an embedding must be possible since $F$ is a face boundary. 

Next, we claim that at most one of the intervals sends an edge to an infinite component of $G - F'$. For if \ta\ intervals $I\neq J$ adjacent with  infinite components $C_I$, $C_J$ of $G - F'$, then replacing $I$ in $F'$ by a path through $C$ we would obtain a cycle $D$ that separates $C_I$ from $C_J$ by \eqref{interv} (\fig{fIJ}). But then  $g^{-1}I,g^{-1}J$ must lie in distinct sides of $g^{-1}D$ since $F=g^{-1}F'$ and $F$ is a face boundary, contradicting the fact that \sig\ is \vapf. 

\epsfxsize=0.6\hsize
\showFig{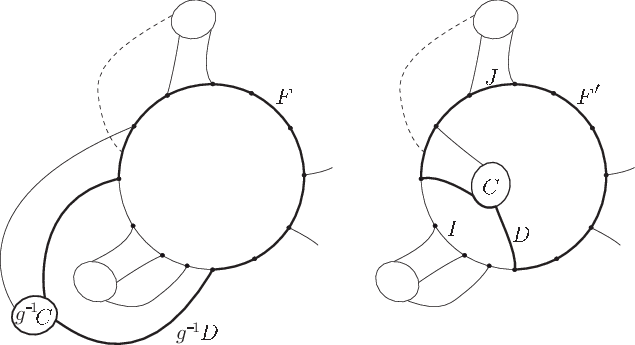}{A contradictory situation in the proof of \Lr{faces}.}

Thus our claim is proved, implying that there is a unique interval $I$ adjacent with the infinite component of $G - F'$. This fact, combined with \eqref{interv}, implies that deleting the vertices $x,y \in N(C)$ bounding $I$ leaves a finite component, namely the component $K$ of $G -\{x,y\}$ containing $F'-I$; note that here we are using the fact that \g has no parallel edges to make sure that $K$ contains at least one vertex. But this contradicts \Lr{sepfin}.
\end{proof}

\Lr{faces} implies that every \vapf\ \Cg\ admits an embedding that is topologically identical  around any vertex. In order to make this more precise we will need a few definitions.

Given an embedding \sig\ of a \Cg\ $G$ with generating set $S$, we consider \fe\ vertex $x$ of \g the embedding of the edges incident with $x$, and define the \defi{spin} of $x$ to be the cyclic order of the set $E_x$ of edges incident with $x$ in which $e$ is a successor of $f$ whenever the edge $e$ comes immediately after the edge $f$ as we move clockwise around $x$. Note that the set $E_x$  depends only on $S$ and  our convention on whether to draw one or two edges for involutions. This allows us to compare spins of different vertices, by identifying edges corresponding to the same generator in $S \cup S^{-1}$. 

Call an edge of \g \defi{spin-preserving} if its two endvertices have the same spin in \sig, and call it \defi{spin-reversing} if the spin of one endvertex can be obtained from the spin of the other by reversing the order. Call a colour in $S$ \defi{\pr} if all edges bearing that colour are  spin-preserving or all edges bearing that colour are spin-reversing in \sig. Finally, call the embedding \sig\ \defi{\pr} if every two vertices have the same spin up to reversing the order, and every colour is \pr\ in \sig. 

It is straightforward to check that \sig\ is \pr\ \iff\ the action on the \Cg\ \g by its group preserves facial walks.

It is known that planar 3-connected \Cg s have a \prem\ \cite{cayley3}, while \Cg s of connectivity 2 do not always admit a \prem\ \cite{DrSeSeCon}. Our next result shows that the latter cannot occur in the \vapf\ case.

\begin{corollary}\label{corspin}
Every \vapf\ planar \Cg\ admits a \prem.
\end{corollary} 

Again, our convention that involutions are represented by a single edge rather than a pair of parallel edges is necessary here. For example, the \Cg\ of $<a,b,c| a^2, b^2, c^2, (ab)^3 , (bc)^3 , (ca)^3 >$ is a hexagonal grid that does not admit a \fapr\ if  its generators are represented by pairs of parallel edges.
\begin{proof}
Let \g be a \Cg\ with an \vapf\ embedding \sig, and let $S$ be its set of generators. We define an equivalence relation $\sim$ on $S \cup S^{-1}$ as follows. Declare two elements to be \defi{neighbours}, if there is a finite face boundary incident with a fixed vertex $o\in V(G)$ containing the two edges corresponding to these elements, and let $\sim$ be the transitive closure of the neighbour relation. By \Lr{faces}, if two edges incident with some other vertex $x$ lie  in a common finite face, then the  corresponding  edges incident with $o$ are also adjacent in the cyclic ordering. Thus neither the neighbour relation nor $\sim$ can depend on the choice of $o$.

Note that, by the definitions, equivalence classes of $\sim$ give rise to consecutive members of the spin of $o$, or any other vertex. This means that the spin of any vertex $x$ can be obtained from that of $o$ by changing the order in which the various $\sim$-classes appear or reversing the order in which the elements of a class appear. 

Our next claim that the edges incident with any vertex $x$ corresponding to each $\sim$-class lie in distinct components of $G - x$; in particular, $G$ is not 2-connected unless  $\sim$ only has one equivalence class. Indeed, between any two $\sim$-classes in the spin of $x$ there must be an infinite face by the definition of the neighbour relation. Now if there is a path $P$ in $G - x$ connecting the other endvertices of two edges $e,f$ incident with $x$ from distinct $\sim$-classes, then attaching $e$ and $f$ to $P$ we would obtain a cycle through $x$ that would separate two such infinite faces, contradicting \vapf ness.

Our last two observations combined show that we can modify our embedding of \G\ into an embedding in which $x$ and $o$ have the same spin up to reflection (i.e.\ reversing the order) by topological operations like reflecting the embedding of a single component of $G - x$ or changing the order in which two such components are embedded around $x$ (in the case where there are more than two of them). Note that such operations can only reverse the order of the spin of $o$ (if $o$ happens to lie in one of the components reflected), but they do not change adjacencies. Thus for every finite vertex set, we can make sure that all vertices in the set have the same spin up to reflection by finitely many such operations. By a standard compactness argument, we obtain an \vapf\ embedding of \g in which all vertices have the same spin as $o$ or its reflection. 

It remains to show that each element of $S$ can be forced to be \pr\ as defined above. For this we distinguish two cases given an $s\in S$. The first case is when each edge of colour $s$ is a bridge, i.e.\ its removal separates \G. In this case we can perform reflecting operations as above to make all such edges spin-preserving. In the other case, an argument similar to the one above shows that for each such edge $e$, at most one of the faces incident with $e$ is infinite in any \vapf\ embedding of \G. Thus $e$ lies in a finite face boundary $C$, and by \Lr{faces} all translates of $C$ are face boundaries as well. It is now easy to see that all translates of $e$ are spin-preserving or they are all spin-reversing, as $C$ forces one of the two behaviours.
\end{proof}

In \Sr{sML} we will prove a result that is, in a sense, the converse of \Cr{corspin}; we will show how to use the ideas of spin and consistency to deduce \vapf ness from properties of a presentation.

\section{\Vapf ness as a group-theoretical invariant} \label{sinv}

A planar group can admit both \vapf\ and non-\vapf\ \Cg s. For example, the \Cg\ corresponding to the presentation $\left<a,b\mid b^2, abab \right>$, of the infinite dihedral group, is \vapf\ planar, but adding the redundant generator $c=ab$ keeps the \Cg\ planar and makes it non-\vapf\ as the reader can check. Thus \vapf ness is not group-theoretical invariant in general. However, it becomes an invariant if one only considers \tcon\ \Cg s:

\begin{theorem} \label{thmtc}
If a group \Gam\ has a \tcon\ \vapf\ planar \Cg\ and a group $\Delta$ has a \tcon\ non-\vapf\ planar \Cg, then \Gam\ is not isomorphic to $\Delta$.
\end{theorem}

Before proving this let us see a further example showing that it is necessary that both graphs in the assertion be \tcon. Consider the \Cg\ corresponding to the presentation $\left<a,b\mid a^4, b^4\right>$. This is a free product of 4-cycles, and it \ises\ that it has an \vapf\ embedding and that its connectivity is 1. Now add the redundant generators $c=ab$ and $d=a^2 b^2 a^2$. Note that $d^2=1$. It is not hard to check that the corresponding \Cg\ is \tcon, and that it is still planar: \fe\ 4-cycle $C$ spanned by $a$, embed the four 4-cycles spanned by $b$ incident with $C$ alternatingly inside and outside $C$. Such an embedding is not \vapf, for $C$ separates two infinite subgraphs. It now follows easily from the following classical result, proved by Whitney \cite[Theorem 11]{whitney_congruent_1932} for finite graphs and by Imrich \cite{ImWhi} for infinite ones, that no embedding of this graph is \vapf.


\begin{theorem} \label{imrcb}
Let \g be a \tcon\ graph embedded in the plane. Then every automorphism of \g maps each facial path to a facial path.
\end{theorem}

We will need a few lemmas for the proof of \Tr{thmtc}.

\begin{lemma} \label{LG3m}
 Let \g be a \iicon\ planar graph and let $\omega, \psi$ be distinct ends of \G. Then there is a cycle $C$ in \g that \defi{separates} $\omega$ from $\psi$, i.e.\ every double-ray with a tail in $\omega$ and a tail $\psi$ has a vertex in $C$.
\end{lemma} 
\begin{proof}
Fix an embedding \sig\ of \G. Consider a finite set of vertices $S = \{s_1, s_2, \ldots, s_k\}$ separating $\omega$ from $\psi$, and let $C_1$ be a cycle containing $s_1, s_2$; such a cycle exists since \g is \iicon. If $C_1$ does not separate $\omega$ from $\psi$ then both ends lie in one of the sides of $C_1$, the outside say. Note that some vertex of $S$ must also lie outside $C_1$, for otherwise every double-ray with a tail in $\omega$ and a tail $\psi$ would have to go through  $C_1$ to meet $S$, contradicting the fact that $C_1$ does not separate $\omega$ from $\psi$. So pick the least index $j$ \st\ $s_j$ lies outside  $C_1$. Now consider two independent paths $P_1,P_2$ from $s_j$ to $C_1$, and let $A$ be the region of $\R^2\sm \{C_1 \cup P_1 \cup P_2\}$ containing rays in \oo. The boundary of $A$ is a cycle  $C_2$ containing $P_1 \cup P_2$ and a subpath of $C_1$. Note that no element of $\{s_1, s_2, \ldots, s_j\}$ lies in $A$ because those points do not lie outside $C_1$. Repeating this argument we construct the sequence of cycles $C_1, C_2, \ldots, C_m$, terminating with a cycle $C_m$ \st\ the  outside of $C_m$ contains \oo\ but none of the $s_i$. This cycle  separates $\omega$ from $\psi$ because every double-ray with a tail in $\omega$ and a tail $\psi$ has to cross it to meet $S$.
\end{proof}

Using this we can prove:

\begin{lemma} \label{LG3}
Let \g be a \iicon\ graph with an \vapf\ embedding \sig\ and more than 1 end. Then at least one of the faces of \sig\ has infinite boundary.
\end{lemma} 
\begin{proof}
By \Lr{LG3m} there is a cycle $C$ separating two ends $\omega$, $\psi$ of \G. Since $\sig$ is \vapf, both these ends lie in the same side of $C$, the outside say. Let $K_\oo$ (respectively $K_\psi$) be the component of $G - C$ containing rays in \oo\ (resp.\ $\psi$). Easily, it is possible to find independent subpaths $P_\oo, P_\psi$ of $C$ \st\ every vertex of $C$ adjacent with $K_\oo$ lies in $P_\oo$, and similarly for $K_\psi$ and $P_\psi$. 
Let $x$ be an endvertex of $P_\oo$; \obda, $x$ is adjacent with $K_\oo$. 

By the choice of $x$ we can choose an edge $e=xy$ with $y\in V(K_\oo)$ and a further edge $f=xz$ incident with $x$ with $z$ not in $K_\oo$ and $f$ not in $P_\oo$, so that $e,f$ lie on a common face boundary $F$, bounding some face $\cf$, say. Note that $f$ may or may not lie on $C$. 
Now if $F$ is infinite we are done, so suppose it is finite. Consider the subpath $F'$ of $F$ starting with the edge $xy$ and finishing at the first visit of $F$ to $C$. Thus one of the endvertices of $F'$ is $x$, and  the other endvertex $x'$ must also lie on $P_\oo$ since, easily, $F'\sm \{x,x'\}$ is contained in the component $K_\oo$ of $G \sm C$. Now consider the cycle $D$ contained in $F' \cup P_\oo$. We claim that $K_\oo,K_\psi$ lie in distinct sides of $D$ which contradicts our assumption that \sig\ is \vapf. 

To see this, note that as $P_\oo \cap D$ joins two vertices $x,x'$ on the face boundary $F$, it defines two regions in $\R^2 \sm \cf$, one region $\ca$ bounded by $D$ and one region $\cb$ bounded by $(F\sm F' )\cup P_\oo$. By the definition of  $K_\oo$, there is a ray in $\oo$ starting at $y$ and avoiding $C$, and so this ray is contained in $\ca$. By inspecting the  cyclic ordering of the edges incident with $x$, it is easy to see that the edge of $C\sm P_\oo$ incident with $x$ (which edge may coincide with $f$) lies in $\cb$ by the choice of $\cf$. Thus, any ray in $\psi$ starting with that edge and avoiding $P_\oo$, which exists by the definition of $K_\oo$, $K_\psi$, lies in $\cb$. This proves our claim that $D$ separates rays in $\oo$ from rays in $\psi$.
\end{proof}

Our last lemma is

\begin{lemma} \label{tc1end}
There is no \tcon\ vertex-transitive \vapf\ planar graph with more than 1 end.
\end{lemma} 
\begin{proof}
If such a graph \g exists, then by \Lr{LG3} it  has an infinite face-boundary. By \Tr{imrcb} this implies that every vertex of \g is incident with an infinite face-boundary.

Thus we can pick two vertices $x,y$ that lie in a common double ray $R$ of \G\ contained in a face-boundary. As \g is \tcon, there are three independent \pths{x}{y}\ $P_1,P_2,P_3$ by Menger's theorem \cite[Theorem~3.3.1]{diestelBook05noEE}. By an easy topological argument, there must be a pair of those paths, say $P_1,P_2$, whose union is a cycle $C$ \st\ some side of $C$ contains a tail of $R$ and the other side of $C$ contains $P_3$. We may assume \obda\ that $P_3$ is not a single edge, for we are allowed to choose $x$ and $y$ far apart. Thus the side of $C$ containing $P_3$ contains at least one vertex $z$. By our previous remarks, $z$ is incident with an infinite face-boundary. This means that both sides of $C$ contain infinitely many vertices, contradicting our assumption that \G\ is \vapf.
\end{proof}

We can now prove the main result of this section.
\begin{proof}[Proof of \Tr{thmtc}]
If any of $\Gam,\Delta$ is 1-ended then we are done since it is well-known, and not hard to prove, that all its planar \Cg s are \vapf\ in this case. The result now follows immediately from \Lr{tc1end}.
\end{proof}

\section{Facial presentations} \label{sML}

In this section we derive a further characterisation of the groups that admit a flat Cayley complex by means of group presentations. This characterisation is motivated by the concept of \prem s introduced before \Cr{corspin}.

Suppose we are given a group presentation $<S\mid R>$, with $S,R$ finite, and a fixed spin $\pi$ on $S$, that is, a cyclic order of $S \cup S^{-1}$ (note that $| S \cup S^{-1}| = 2|S| - |B|$, where $B\subseteq S$ is the set of $b\in S$ with $b=b^{-1}$, i.e.\ the set of involutions). Moreover, we fix an assignment $f: S \to \{0,1\}$, and say that $s \in S$ is spin-preserving if $f(s)=0$ and spin-reversing if $f(s)=1$. Let $T(S)$ be the \Cg\ corresponding to the presentation $<S\mid \{b^2 \mid b\in B\}>$, with parallel edges corresponding to the elements of $B$ replaced by single, undirected edges, and note that $T(S)$ is a tree.
Easily, $T(S)$ has a \prem\ $\tau$ in which the spin of each vertex is either $\pi$ or its reversal, and each $s\in S$ is spin-preserving \iff\ $f(s)=0$. Now call our presentation $<S\mid R>$ facial \wrt\ the data $\pi, f$, if for every rotation $w'$ of every word $w \in R$,  and any vertex $t$ of $T(S)$, the walk on $T(S)$ that starts at $t$ and is induced by $w'$ is facial in $\tau$.

It is a good exercise to try prove that each relator of a \fapr\ contains an even number of occurences of spin-reversing generators unless we are in the rather trivial case where $|S|=1$.

Note that every facial walk consisting of two edges of $T(S)$, or every two elements of $S \cup S^{-1}$ that are adjacent in $\pi$, uniquely determine a 2-way infinite, periodic, `facial' word. This easily implies that \ti\ a canonical way to rewrite any \fapr\ as  $<S \mid E_1^{r_1}, \ldots E_k^{r_k}> $, where $E_i$ is aperiodic and each 2-way infinite facial walk in $T(S)$ is obtained by repeatedly reading one of the $E_i$. Moreover, we have $k\leq |S\cup S^{-1}|$, but $k$ can be as small as 1 even if $S$ is large and all faces are finite; consider for example the presentation $< a,b,c \mid abcbac>$ which is facial \wrt\ the spin $a, c^{-1}, b, a^{-1}, c, b^{-1}$ and all edges spin-preserving.

We can now formulate the main result of this section, which complements \Cr{corspin}:

\begin{theorem}\label{thmfac}
The \Cg\ corresponding to any \fapr\ is planar and admits a consistent \vapf\ embedding.
\end{theorem} 

As an example application, consider a Coxeter presentation $$<s_1, \ldots , s_k \mid s_1^2, \ldots , s_k^2, (s_1 s_2)^{r_{12}}, (s_2 s_3)^{r_{23}}, \ldots , (s_k s_1)^{r_{k1}}>$$ with all exponents $r_{ij}$ at least $2$ and possibly infinite. It is straightforward to check that every such presentation is facial \wrt\ to the spin $s_1, \ldots , s_k$ and all generators spin-reversing. Thus \Tr{thmfac} tells us that the corresponding \Cg\ is planar and \vapf\ (this fact is probably well-known to experts in geometry).

For the proof of \Tr{thmfac} we will use 
the following theorem of Thomassen, which generalises MacLane's classical planarity criterion \cite[Theorem~4.5.1]{diestelBook05noEE} to infinite \vapf\ planar graphs. 
A \defi{2-basis} of \g is a generating set $B$ of the cycle space \ccfg\ \st\ no edge of \g appears in more than two elements of $B$. 

\begin{theorem}[{\cite[Section 7]{thoPla}}] \label{thom}
A 2-connected graph has a 2-basis if and only if it is planar and has an \vapf\ embedding.
\end{theorem} 

The requirement that \g be 2-connected is essential in this assertion: consider for example the \Cg\ $G$ corresponding to the presentation  $<a_1,a_2, z \mid  a_1 a_2= a_2 a_1>$. Thus \g is the free product of the square grid with the integer line. Note that the squares of the former factor form a 2-basis of \G, still $G$ does not have an \vapf\ embedding. 

In order to be able to still apply \Tr{thom} in our setup, we will use the following fact. We say that a group presentation $A$ \defi{contains} a group presentation $B$, if the \Cg\ corresponding to $B$ is a subgraph of the \Cg\ corresponding to $A$. Note that this means that the generating set of $A$ contains that of $B$, but the sets of relators can be quite different.

\begin{lemma} \label{2conlem}
Every \fapr\ is contained in a \fapr\ the \Cg\ of which is  2-connected.
\end{lemma}
\begin{proof}
Let $<S \mid R>$ be a \fapr\ \wrt\ a spin $\pi$ and an assignment $f$ as in the above definition. If its \Cg\ \g is 2-connected we are done, so suppose it is not. Then any vertex $x$ separates \G, and so we can find $s,t\in S\cup S^{-1}$ that are consecutive in the spin $\pi$ but $xs$ and $xt$ lie in distinct component of $\g - x$. We now construct a new presentation $<S' \mid R'>$ containing $<S \mid R>$ as follows. Firstly, we add a new generator $z$ to $S$ to obtain $S'$. Secondly, add the relator $z=t^{-1}s$ to $R$. Finally, for every $w \in R$ containing $t^{-1}s$ (respectively $s^{-1}t$) as a subword ---assume here that $w$ is spelt without using exponents other than $-1$--- replace that subword by the letter $z$ (resp.\ $z^{-1}$). Let $R'$ be the set of relators obtained after all these changes.

It is easy to see that $<S' \mid R'>$ contains $<S \mid R>$, as it amounts to adding a redundant generator $z$. It is also straightforward to see that the data $\pi, f$ can be extended so as to make $<S' \mid R'>$ a \fapr. Indeed, we can let $f(z)= f(s) XOR f(t)$. To extend $\pi$ to $S'$, let us assume \obda\ that $s$ immediately precedes $t$ in $\pi$. If $s$ is spin-preserving, i.e.\ if $f(s)=0$, then we insert $z^{-1}$ into $\pi$ at the position just before $s$. Otherwise,  we insert $z^{-1}$ into $\pi$ at the position just after $s$. Similarly, we insert $z$ at the position just after (respectively, before) $t$
if $t$ is spin-preserving (resp.\ spin-reversing). It is now straightforward to check that every relator in $R'$ is facial \wrt\ this data.
\end{proof}

We can now prove \Tr{thmfac}.
\begin{proof}
Let $<S \mid R>$ be a \fapr. By our last lemma, we can find a \fapr\ $<S' \mid R'>$ containing $<S \mid R>$ the \Cg\ $G'$ of which is  2-connected. Let us show that $G'$ admits an \vapf\ embedding.

Recall that $G'$ admits a presentation of the form
$$<S \mid E_1^{r_1}, \ldots E_k^{r_k}>,$$
where each $E_i$ is an aperiodic facial word. We may assume \obda\ that $r_i$ is minimal with the property that $E_i^{r_1}$ induces a closed walk in $G'$, for otherwise we can replace $r_i$ with some smaller value in the $i$th relator and obtain an equivalent presentation.

It would make our proof simpler if \fe\ $s\in S \cup S^{-1}$ and every  $E_i$, the letter $s$ appears at most once in $E_i$. This however is not always the case: the presentation $<c,b \mid cbcb^{-1}>$ for example, the \Cg\ of which is a square grid, is facial \wrt\ to the spin $c,b,c^{-1},b^{-1}$, with $c$ being spin-preserving and $b$ spin-reversing; but $c$ appears twice in the word $cbcb^{-1}$ (in this section we consider $s$ and $s^{-1}$ to be distinct letters). Still, we can easily modify our presentation whenever this situation occurs, to ensure that each letter $s$ appears at most once in each $E_i$. To begin with, note that $s$ can appear at most twice in $E_i$: for as $E_i$ is uniquely determined (up to rotation and reversing) by any letter and `side', if $E_i$ contains three occurrences of $s$, then two of them will correspond to the same side, implying that $E_i$ is periodic contrary to our assumption. By the same arguments, if $s$ appears twice in $E_i$, then this means that some rotation of the word $E_i$ is read along both `sides' of $s$ in $T(S)$. To avoid this situation, we can extend $S$ by a new generator $s'$, and add the relator $s's^{-1}$ to our presentation. The new presentation yields the same \Cg\ with a parallel edge added to each $s$ edge, and it is straightforward to amend the spin data to make sure that the presentation is still facial. Thus from now on we will assume that 
\labtequ{once}{\fe\ $s\in S \cup S^{-1}$ and every  $E_i$, the letter $s$ appears at most once in $E_i$.}
It is still possible though that $E_i$ contains both $s$ and $s^{-1}$.

Let $W$ be the set of walks in $G'$ induced by the above relators $E_i^{r_i}$.
Note that \eqref{once} and our choice of the $r_i$ imply that 
\labtequ{notwo}{no walk in $W$ traverses any edge of $G'$ twice in the same direction.}
For if this was the case, then the subwalk between to subsequent visits to the first endpoint of that edge would be closed, and by \eqref{once} it would be induced by (a rotation of) $E_i^{m}$ with $m< r_i$.

Recall that  $W$ generates $\ccfh(G')$ (\Lr{relcc}). The idea is to try apply \Tr{thom} to $W$, adapting the fact that every edge of a planar graph appears in just two facial walks to our situation. 

Before we do that, we first simplify $W$ as follows. \Fe\ walk $w\in W$ traversing some edge $e$ of $G'$ in both directions, we split $w$ into two closed walks $w_1, w_2$ that traverse $e$ less often in total in such a way that $w_1+w_2$ corresponds to the same element of $\ccfh(G')$ as $w$. We repeat this recursively as often as needed until no walk traverses an edge in both directions. Finally, if two of the resulting closed walks can be obtained from one another by rotation or inversion, we delete one of them, and repeat until no such pairs exist. Let $W'$ denote the resulting set of walks.

By construction, $W'$ still  generates $\ccfh(G')$. 
We claim that $W'$ has the desired property that every edge of $G'$ appears in at most two elements of $W'$, and at most once in each of them. 


Indeed, since each element of $W$ traverses each edge at most once in each direction by \eqref{notwo}, each element of $W$ traverses each edge at most once in total.
Next, suppose that three distinct walks in $W'$ traverse some edge $e$. Again by \eqref{notwo}, no two of them come from the same element of $W$. Then, as our presentation is planar, two of them, call them $w_1,w_2$, are induced by the same relator $E_i^{r_i}$. Thus each of $w_1,w_2$ contains the same number of edges of the colour of $e$ and, by \eqref{once}, the same subword of $E_i^{r_i}$ (which must be a rotation of the word $E_i$) is read between any two subsequent traversals of such an edge. This easily implies that $w_1$ is a rotation of $w_2$, contradicting the construction of $W'$.

This proves that each edge appears at most twice in $W'$. Since $W'$ generates $\ccfh(G')$, the set of edge-sets of its elements generates $\ccf(G')$\footnote{It is easy to see that the canonical projection of a generating set of $\ccfh(G)$ to $\ccf(G')$ generates \ccfg; the converse is not always true \cite[Figure 9]{kavitha_cycle_2009}.}.  Splitting each such edge-set into edge-disjoint cycles ---it is well-known that this is possible \cite[Proposition 1.9.2.]{diestelBook05noEE}--- we obtain a 2-basis of $\ccf(G')$. By \Tr{thom}, $G'$ admits an \vapf\ embedding, and so does its subgraph \G. 

By \Cr{corspin}, \G\ even admits a consistent \vapf\ embedding.

\end{proof}

Using the results of \Sr{smain} we can prove now that the converse of \Tr{thmfac} is also true, yielding \Cr{corfapr}; we repeat its statement here.
\begin{nocorollary}
A \Cg\ admits an \vapf\ embedding \iff\ it admits a \fapr.
\end{nocorollary} 
\begin{proof}
Let \g be a \Cg\ with an \vapf\ embedding. Then \g admits a consistent  \vapf\ embedding by \Cr{corspin} \sig. As in the first part of the proof of \Tr{thmvapf}, the set  $R$ of relations corresponding to the finite facial closed walks of \sig\ incident with the group identity yields a presentation of \G, and this presentation is, by construction, facial \wrt\ to the spin data of \sig.
\end{proof}

Note that this implies that every group admitting an \vapf\ planar \Cg\ is finitely presented. This fact extends to all planar groups \cite{droInf}. In the \vapf\ case $|S\cup S^{-1}|$ is an upper bound on the number of relators needed to present a group with generating set $S$, but in the general case this is not necessarily the case; see \cite[Problem 10.2.]{cayley3}.

\medskip
{\bf Remark:} One could modify the definition of a \fapr\ by not giving involutions in $S$ any special treatment, that is, by letting $T(S)$ be the \Cg\ of the presentation $<S\mid \emptyset>$ (a tree of degree $2|S|$). \Tr{thmfac} would then still be true by the same proof, but I suspect that its converse in \Cr{corfapr} would fail; see \fig{ladder}.

\medskip

\comment{
	\begin{corollary} \label{thomcor}
	A connected graph has a 2-basis if and only if it is planar and each of its 		blocks has an \vapf\ embedding.
	\end{corollary} 
}

It is known that the groups admitting an \vapf\ planar \Cg\ can be effectively enumerated \cite{droInf}. Using \Cr{corfapr} we can strengthen this as follows.
\begin{corollary} \label{efen}
The \vapf\ planar \Cg s can be effectively enumerated.
\end{corollary}
\begin{proof}
It is easy to construct an algorithm that given an abstract group presentation $<S \mid R>$, with both $S,R$ finite, decides whether this presentation is facial \wrt\ some spin data, since there are only finitely many possibilities for such data. The assertion thus follows from \Cr{corfapr}.
\end{proof}

\acknowledgements{I am grateful to Martin Dunwoody for pointing out a shortcoming of an earlier version of the paper.}

\bibliographystyle{plain}
\bibliography{collective}
\end{document}